\documentclass[12pt]{amsart}
\usepackage[hidelinks]{hyperref}
\usepackage{amsmath,amssymb,amsthm,mathtools}
\usepackage{thmtools,thm-restate}
\usepackage{tikz-cd}
\usepackage{tikz}
\usepackage{enumitem}
\usepackage{graphicx}
\usetikzlibrary{matrix,arrows.meta,bending}
\usepackage{color}
\usepackage[numbers]{natbib}
\usepackage{booktabs}
\usepackage{longtable}
\usepackage{pgfplots}
\pgfplotsset{compat=1.14} 
\usetikzlibrary{arrows.meta,decorations.pathreplacing,positioning,calc} 

\usepackage[a4paper, left=3cm, right=3cm, top=3cm, bottom=3cm, footskip=15mm]{geometry}

\newtheorem{theorem}{Theorem}[section]

\newtheorem{lemma}[theorem]{Lemma}
\newtheorem{proposition}[theorem]{Proposition}

\newtheorem{conjecture}[theorem]{Conjecture}
\theoremstyle{definition}
\newtheorem{definition}[theorem]{Definition}

\theoremstyle{remark}
\newtheorem{remark}[theorem]{Remark}

\numberwithin{equation}{section}

\makeindex

\usepackage{fancyhdr}
\usepackage{calc} 

\AtBeginDocument{%
  \setlength{\textwidth}{\dimexpr\paperwidth - 6cm\relax}%
  \setlength{\oddsidemargin}{\dimexpr 3cm - 1in\relax}%
  \setlength{\evensidemargin}{\dimexpr 3cm - 1in\relax}%
  \setlength{\marginparwidth}{2.5cm}%
}

\begin{document}

\title{New bounds for the Heilbronn triangle problem}

\author{T. Agama}
\address{Department of Mathematics, African Institute for Mathematical science, Ghana
}
\email{theophilus@aims.edu.gh/emperordagama@yahoo.com}


\date{\today}



\begin{abstract}
Using ideas from the geometry of compression, we improve on the current upper and lower bounds of the Heilbronn triangle problem. In particular, let $\Delta(s)$ denote the minimal area of the triangle induced by $s$ points on a unit disk. We have the upper bound 
$$
\Delta(s)\ll \frac{1}{s^{\frac{3}{2}-\epsilon}}
$$ 
for small $\epsilon:=\epsilon(s)>0$ and the lower bound
$$
\Delta(s)\gg \frac{\log s}{s\sqrt{s}}.
$$
\end{abstract}

\maketitle

\section{Introduction}
\label{sec:intro}

The Heilbronn triangle problem asks for the maximal possible value of the minimal area of a triangle determined by $s$ points placed inside a convex planar region (most commonly the unit disk). Let $\Delta(s)$ denote the supremum, taken over all $s$-point subsets of the unit disk, of the minimal triangle area determined by the set. Heilbronn originally conjectured the extremal behaviour
$$
\Delta(s)=O\big(s^{-2}\big),
$$
which, if true, would indicate that well-distributed configurations force triangles of area on the order of $s^{-2}$.\\

A number of classical and more recent results establish nontrivial lower and upper bounds for $\Delta(s)$. Early analytic approaches produced the first nontrivial upper bounds (see, e.g., \cite{roth1951problem,roth1976developments}), while explicit combinatorial constructions yield lower bounds that rule out too-rapid decay of $\Delta(s)$.  In particular, the celebrated construction of Koml\'os--Pintz--Szemer\'edi provides a configuration demonstrating the improved lower bound
$$
\Delta(s)\gg \frac{\log s}{s^{2}},
$$
which shows that the logarithmic factor naturally appears in extremal constructions \cite{komlos1982lower}. More recent combinatorial refinements have continued to chip away at the upper bound; see, for instance, the recent work of Cohen--Pohoata--Zakharov which introduces new combinatorial ingredients and obtains further improvements \cite{cohen2023new}. For an exposition and algorithmic perspective (including discussions of Erd\H{o}s-style constructions) see \cite{lefmann2002deterministic} and the MathWorld overview \cite{weisstein2003heilbronn}.\\

In this article, we introduce a geometric framework-the \emph{geometry of compression}-which provides a unified language to quantify local clustering and dispersion of planar point sets. The central object is the compression map
$$
\mathbb{V}_m:(x_1,\dots,x_n)\mapsto \left(\frac{m}{x_1},\dots,\frac{m}{x_n}\right),
$$
together with two associated invariants: the \emph{mass of compression} and the \emph{compression gap} $\mathcal{G}\circ\mathbb{V}_m$. The compression gap can be viewed as an induced local radius around a point that effectively measures how close other points may be placed without producing an unusually small triangle. By studying packing and nesting properties of the compression-induced balls one converts combinatorial sums into geometric area estimates that are well-suited for both upper and lower bound arguments (see Proposition~\ref{cgidentity}, Lemma~\ref{gapestimate} and Proposition~\ref{circle area} below).\\

Informally, our main results are an improved upper bound and a constructive lower bound:
\begin{itemize}
  \item an upper bound of the form
    $$
      \Delta(s)\ll s^{-3/2+\epsilon(s)},
    $$
    where $\epsilon(s)>0$ is a slowly varying function (specified precisely in Section~\ref{sec:upper}); and
    \bigskip
    
  \item a constructive lower bound
    $$
      \Delta(s)\gg \frac{\log s}{s\sqrt{s}},
    $$
    obtained by placing admissible points on compression-induced circles and estimating the areas of the triangles determined by adjacent points (see Section~\ref{sec:lower}).
\end{itemize}
\bigskip

The upper bound is proved by partitioning any $s$-point configuration into compression-balls and estimating the total area available for hosting such balls; a pigeonhole/counting argument then forces the existence of a triangle of area at most the stated order. The lower bound is constructive: we explicitly place points on a compression-circle with equidistant chords, and then use the compression mass and harmonic-type estimates to obtain the stated logarithmic factor (see Proposition~\ref{crucial} and Lemma~\ref{elementary}).\\

Technically, the compression framework complements existing analytic and combinatorial methods: it gives a geometric handle on local structure while translating certain combinatorial sums into geometric packing estimates that are comparatively straightforward to manipulate. 

\subsection{Organization of the paper} The remainder of the paper is organized as follows.  In Section~\ref{sec:compression}, we introduce the compression formalism (definitions, basic properties, and key identities). Section~\ref{sec:upper} contains the proof of the upper bound, and section~\ref{sec:lower} gives the constructive lower bound.
\bigskip

\subsection{Notations}
The following notations are standard
\bigskip

\begin{itemize}
    \item The notation $f(n)\ll g(n)$ or $f(n)=O(g(n))$ means that there exist an absolute constant $c>0$ and some fixed real number $n_0>0$ such that $|f(n)|\leq c|g(n)|$ for all $n\geq n_0$. In the case where the constant $c$ depends on a fixed parameter (say $d$), we may write $f(n)\ll_d g(n)$.
\bigskip

   \item The notation $f(n)\gg g(n)$ means that there exist an absolute constant $c>0$ and some fixed real number $n_0>0$ such that $|f(n)|\geq c|g(n)|$ for all $n\geq n_0$. In the case where the constant $c$ depends on a fixed parameter (say $d$), we may write $f(n)\gg_d g(n)$.
\bigskip

   \item The notation $f(n)\asymp g(n)$ means that there exist absolute constants $c_1,c_2>0$ and some fixed real number $n_0$ such that $c_1|g(n)|\leq f(n)\leq c_2|g(n)|$ for all $n\geq n_0$. In the case where the constant $c_1,c_2$ depends on a fixed parameter (say $d$), we may write $f(n)\asymp_d g(n)$.
\bigskip

   \item The symbols $\epsilon$ and $\delta$ should be understood as positive quantities whose magnitudes can be made arbitrarily small.
\bigskip

   \item The notation $f(n)\sim g(n)$ means that $\lim \limits_{n\longrightarrow \infty}\frac{f(n)}{g(n)}=1$.
\bigskip

    \item The notation $f(n)=o(1)$ means that $\lim \limits_{n\longrightarrow \infty}f(n)=0$.
\end{itemize}
\bigskip

Let $\mathcal{D}$ denote any convex shape in the plane, and let $\Delta(S)$ represent the minimal area of the triangle induced by a set of $s$ points in $\mathcal{D}$. Define $\Delta(s)$ as the supremum of all such $\Delta(S)$. Heilbronn conjectured, in what is now known as the Heilbronn triangle problem, that:

\begin{conjecture}
The minimal area of the triangle induced by $s$ points in $\mathcal{D}$ satisfies
\begin{align}
\Delta(s)=O\left(\frac{1}{s^2}\right).\nonumber
\end{align}
\end{conjecture}

Paul Erd\H{o}s previously established a lower bound related to Heilbronn's conjecture, showing that
\begin{align}
\Delta(s) \gg \frac{1}{s^2}.\nonumber
\end{align}

If had been proven true, this lower bound would have confirmed the Heilbronn conjectured upper bound as the sharpest possible. The Heilbronn triangle problem had remained unsolved for a long time until a significant breakthrough in 1982 by Komlos, Pintz, and Szemeredi, who constructed a set of points in $\mathcal{D}$ with a minimal induced triangle area, $\Delta(s)$, satisfying the lower bound
\begin{align}
\Delta(s) \gg \frac{\log s}{s^2}.\nonumber
\end{align}

The asymptotic growth rate of the minimal area of the triangle determined by a finite set of points in $\mathcal{D}$ remains an open question. Thus, the pursuit of improved lower and upper bounds continues to be a valuable area of research. The first non-trivial upper bound was obtained by Roth, who showed that

\begin{align}
\Delta(s) \ll \frac{1}{s\sqrt{\log \log s}}.\nonumber
\end{align}

A refinement of Roth's method eventually yielded the best currently known upper bound, as shown by Koml\'os, Pintz, and Szemer\'edi:

\begin{align}
\Delta(s) \ll \frac{e^{c\sqrt{\log s}}}{s^{\frac{8}{7}}}.\nonumber
\end{align}
\bigskip

A new upper bound of the form $$\Delta(s)\ll \frac{1}{s^{\frac{8}{7}+\frac{1}{2000}}}$$ has recently appeared in \cite{cohen2023new}. Using a completely new idea from the geometry of compression, we obtain an improved upper bound

\begin{theorem}
Let $\Delta(s)$ denote the minimal area of the triangle formed by the $s$ points on the unit disk. We have the upper bound 
\begin{align}
\Delta(s)\ll \frac{1}{s^{\frac{3}{2}-\epsilon}}\nonumber
\end{align}
for small $\epsilon:=\epsilon(s)>0.$
\end{theorem} 

We also obtain a constructive lower bound

\begin{theorem}
Let $\Delta(s)$ denote the minimal area of the triangle formed by the $s$ points on the unit disk. We have the lower bound 
\begin{align}
\Delta(s)\gg \frac{\log s}{s\sqrt{s}}.\nonumber
\end{align}
\end{theorem}  
\bigskip

\section{The geometry of compression}\label{sec:compression}

In this section, we introduce and develop the geometry of \emph{compression} of points in space $\mathbb{R}^n$.

\begin{definition}
By the compression of the scale $1\geq m>0$~($m\in\mathbb{R}$) fixed on $\mathbb{R}^{n}$, we mean the map $\mathbb{V}:(\mathbb{R}\setminus 0)^n\longrightarrow \mathbb{R}^n$ such that 
\begin{align}
\mathbb{V}_m[(x_1,x_2,\ldots, x_n)]=\bigg(\frac{m}{x_1},\frac{m}{x_2},\ldots, \frac{m}{x_n}\bigg)\nonumber
\end{align}
for $n\geq 2$ and with $x_i\neq x_j$ for $i\neq j$ and  $x_i\neq 0$ for all $i=1,\ldots,n$. 
\end{definition}

\begin{remark}
The notion of compression is, in some way, the process of rescaling points in $\mathbb{R}^n$ for $n\geq 2$. Thus, it is important to observe that a compression roughly speaking pushes points very close to the origin away from the origin by a certain scale and similarly draws points away from the origin close to the origin. Intuitively, compression induces some kind of motion on points in the Euclidean space $\mathbb{R}^n$ for $n\geq 2$.
\end{remark}

\begin{proposition}
A compression $\mathbb{V}_m:(\mathbb{R}\setminus 0)^n\longrightarrow \mathbb{R}^n$ of scale $m$ with $1\geq m>0$ is a bijective map.
\end{proposition}

\begin{proof}
Suppose that $\mathbb{V}_m[(x_1,x_2,\ldots, x_n)]=\mathbb{V}_m[(y_1,y_2,\ldots,y_n)]$. It follows that
\begin{align}
\bigg(\frac{m}{x_1},\frac{m}{x_2},\ldots,\frac{m}{x_n}\bigg)=\bigg(\frac{m}{y_1},\frac{m}{y_2},\ldots,\frac{m}{y_n}\bigg).\nonumber
\end{align}
It follows that $x_i=y_i$ for each $i=1,2,\ldots, n$. Surjectivity follows by the definition of the map. Thus, the map is bijective.
\end{proof}

\subsection{The mass of compression}

In this section, we recall the notion of the mass of compression on points in space and study the associated statistics.

\begin{definition}\label{mass}
By the \emph{mass} of a compression of scale $m$ with fixed $1\geq m>0$, we mean the map $\mathcal{M}:\mathbb{R}^n\longrightarrow \mathbb{R}$ such that 
\begin{align}
\mathcal{M}(\mathbb{V}_m[(x_1,x_2,\ldots,x_n)])=\sum \limits_{i=1}^{n}\frac{m}{x_i}.\nonumber
\end{align}
\end{definition}
\bigskip

It is important to note that the condition $x_i\neq x_j$ for $(x_1,x_2,\ldots,x_n)\in \mathbb{R}^n$ is not only a quantifier but a requirement; otherwise, the statement for the mass of compression will be completely flawed. To demonstrate, if we take $x_1=x_2=\cdots=x_n$, then $\mathrm{inf}(x_j)=\mathrm{sup}(x_j)$. In this case, the compression mass of the scale $m$ satisfies 

\begin{align}
m\sum \limits_{k=0}^{n-1}\frac{1}{\mathrm{inf}(x_j)-k}\leq \mathcal{M}(\mathbb{V}_m[(x_1,x_2,\ldots,x_n)])\leq m\sum \limits_{k=0}^{n-1}\frac{1}{\mathrm{inf}(x_j)+k}\nonumber
\end{align}
and this inequality cannot hold. Thus, to ensure that the bounds are tight for the purpose of our work, we will require that the choice of tuple $(x_1,x_2,\ldots,x_n)\in \mathbb{R}^n$ satisfies the requirement: $x_i\neq x_j$ for all $1\leq i,j\leq n$. Hence, in this paper this condition will be highly extolled. In situations where it is not mentioned, it will be assumed that the tuple $(x_1,x_2,\ldots,x_n)\in \mathbb{R}^n$ is such that $x_i\neq x_j$ for $1\leq i,j\leq n$.
 
\begin{lemma}\label{elementary}
We have 
\begin{align}
\sum \limits_{n\leq x}\frac{1}{n}=\log x+\gamma+O\bigg(\frac{1}{x}\bigg)\nonumber
\end{align}
where $\gamma=0.5772\cdots $.
\end{lemma}

\begin{remark}
Next, we prove the upper and lower bounds for the mass of the compression of scale $m$ with fixed $1\geq m>0$.
\end{remark}

\begin{proposition}\label{crucial}
Let $(x_1,x_2,\ldots,x_n)\in \mathbb{R}^n$ with $x_i\neq 0$ for each $1\leq i\leq n$ and $x_i\neq x_j$ for $i\neq j$. We have
\begin{align}
m\log \bigg(1-\frac{n-1}{\mathrm{sup}(x_j)}\bigg)^{-1} \ll \mathcal{M}(\mathbb{V}_m[(x_1,x_2,\ldots, x_n)])\ll m\log \bigg(1+\frac{n-1}{\mathrm{inf}(x_j)}\bigg)\nonumber
\end{align}
for $n\geq 2$.
\end{proposition}

\begin{proof}
Let $(x_1,x_2,\ldots,x_n)\in \mathbb{R}^n$ for $n\geq 2$ with $x_j\neq 0$. It follows that 
\begin{align}
\mathcal{M}(\mathbb{V}_m[(x_1,x_2,\ldots, x_n)])&=m\sum \limits_{j=1}^{n}\frac{1}{x_j}\nonumber \\&\leq m\sum \limits_{k=0}^{n-1}\frac{1}{\mathrm{inf}(x_j)+k}\nonumber
\end{align}
and the upper bound can be deduced for this sum. The lower bound can also be deduced in the following way
\begin{align}
\mathcal{M}(\mathbb{V}_m[(x_1,x_2,\ldots, x_n)])&=m\sum \limits_{j=1}^{n}\frac{1}{x_j}\nonumber \\&\geq m\sum \limits_{k=0}^{n-1}\frac{1}{\mathrm{sup}(x_j)-k}.\nonumber
\end{align}
\end{proof}

\begin{definition}\label{gap}
Let $(x_1,x_2,\ldots, x_n)\in \mathbb{R}^n$ with $x_i\neq 0$ for all $i=1,2\ldots,n$. By the \emph{compression gap} of compression $\mathbb{V}_m$ of the scale $m>0$, denoted $\mathcal{G}\circ \mathbb{V}_m[(x_1,x_2,\ldots, x_n)]$, we mean the expression 
\begin{align}
\mathcal{G}\circ \mathbb{V}_m[(x_1,x_2,\ldots, x_n)]=\bigg|\bigg|\bigg(x_1-\frac{m}{x_1},x_2-\frac{m}{x_2},\ldots,x_n-\frac{m}{x_n}\bigg)\bigg|\bigg|\nonumber
\end{align}
\end{definition}

\subsection{The ball induced by compression}

In this section, we introduce the notion of the ball induced by a point $(x_1,x_2,\ldots,x_n)\in \mathbb{R}^n$ under compression of a given scale. We launch in a formal way the following language.

\begin{definition}
Let $(x_1,x_2,\ldots,x_n)\in \mathbb{R}^n$ with $x_i\neq x_j$ for all $1\leq i<j\leq n$ and $x_i\neq 0$ for all $1\leq i\leq n$. By the ball induced by $(x_1,x_2,\ldots,x_n)\in \mathbb{R}^n$ under compression of the scale $m$ with $1\geq m>0$, denoted $\mathcal{B}_{\frac{1}{2}\mathcal{G}\circ \mathbb{V}_m[(x_1,x_2,\ldots, x_n)]}[(x_1,x_2,\ldots,x_n)]$, we mean the set of points $\vec{y}$ satisfying the inequality
\begin{align}
\left|\left|\vec{y}-\frac{1}{2}\bigg(x_1+\frac{m}{x_1},x_2+\frac{m}{x_2},\ldots,x_n+\frac{m}{x_n}\bigg)\right|\right|<\frac{1}{2}\mathcal{G}\circ \mathbb{V}_m[(x_1,x_2,\ldots, x_n)].\nonumber
\end{align}
A point $\vec{z}=(z_1,z_2,\ldots,z_n)\in \mathcal{B}_{\frac{1}{2}\mathcal{G}\circ \mathbb{V}_m[(x_1,x_2,\ldots, x_n)]}[(x_1,x_2,\ldots,x_n)]$ if it satisfies the inequality.
\end{definition}

\begin{remark}
Next, we prove that smaller balls induced by points should essentially be contained in bigger balls. We state and prove this statement in the following result.
\end{remark}

In the geometry of balls induced under compression of fixed scale $m>0$, we will implicitly assume that 
$$
0<m\leq 1.
$$

For simplicity, on occasion, we will choose to write the ball induced by the point $\vec{x}=(x_1,x_2,\ldots,x_n)$ under compression as 
\begin{align}
\mathcal{B}_{\frac{1}{2}\mathcal{G}\circ \mathbb{V}_m[\vec{x}]}[\vec{x}].\nonumber
\end{align}
We will adopt this notation to save enough work space in many circumstances. We find the following estimates for the compression gap useful.

\begin{proposition}\label{cgidentity}
Let $(x_1,x_2,\ldots, x_n)\in \mathbb{R}^n$ for $n\geq 2$ with $x_j\neq 0$ for $j=1,\ldots,n$. We have 
\begin{align}
\mathcal{G}\circ \mathbb{V}_m[(x_1,x_2,\ldots, x_n)]^2=\mathcal{M}\circ \mathbb{V}_1\bigg[\bigg(\frac{1}{x_1^2},\ldots,\frac{1}{x_n^2}\bigg)\bigg]+m^2\mathcal{M}\circ \mathbb{V}_1[(x_1^2,\ldots,x_n^2)]-2mn.\nonumber
\end{align}
In particular, if $m=m(n)=o(1)$ as $n\longrightarrow \infty$, then
\begin{align}
\mathcal{G}\circ \mathbb{V}_m[(x_1,x_2,\ldots, x_n)]^2=\mathcal{M}\circ \mathbb{V}_1\bigg[\bigg(\frac{1}{x_1^2},\ldots,\frac{1}{x_n^2}\bigg)\bigg]-2mn+O\bigg(m^2\mathcal{M}\circ \mathbb{V}_1[(x_1^2,\ldots,x_n^2)]\bigg)\nonumber
\end{align}
for $\vec{x}\in \mathbb{R}^n$ with $x_i\geq 1$ for each $1\leq i\leq n.$
\end{proposition}
\bigskip

Proposition \ref{cgidentity} offers an extremely useful identity. It allows us to pass from the compression gap of the compression on points to the relative distance to the origin. It suggests that points under compression with a large gap must be far away from the origin than points with a relatively smaller gap under compression. That is, the inequality 
\begin{align}
\mathcal{G}\circ \mathbb{V}_m[\vec{x}]< \mathcal{G}\circ \mathbb{V}_m[\vec{y}]\nonumber
\end{align}
with $m:=m(n)=o(1)$ as $n\longrightarrow \infty$ if and only if $||\vec{x}||\lesssim ||\vec{y}||$ for $\vec{x}, \vec{y}\in \mathbb{R}^n$ with $x_i\geq 1$ for all $1\leq i\leq n$. This important transference principle will be used to deduce our results. In particular, we note that in the latter case, we can write the asymptotic 
$$
\mathcal{G}\circ \mathbb{V}_m[(x_1,x_2,\ldots, x_n)]^2\sim \mathcal{M}\circ \mathbb{V}_1\bigg[\bigg(\frac{1}{x_1^2},\ldots,\frac{1}{x_n^2}\bigg)\bigg]=||\vec{x}||^2.
$$

\begin{lemma}[Compression estimate]\label{gapestimate}
Let $(x_1,x_2,\ldots, x_n)\in \mathbb{R}^n$ for $n\geq 2$ with $x_i\geq 1$ for all $1\leq i\leq n$ with $x_i\neq x_j$~($i\neq j$). If $m:=m(n)=o(1)$ as $n\longrightarrow \infty$. We have 
\begin{align}
\mathcal{G}\circ \mathbb{V}_m[(x_1,x_2,\ldots, x_n)]^2\ll n\mathrm{sup}(x_j^2)+m^2\log \bigg(1+\frac{n-1}{\mathrm{inf}(x_j)^2}\bigg)-2mn\nonumber
\end{align} 
and 
\begin{align}
\mathcal{G}\circ \mathbb{V}_m[(x_1,x_2,\ldots, x_n)]^2\gg n\mathrm{inf}(x_j^2)+m^2\log \bigg(1-\frac{n-1}{\mathrm{sup}(x_j^2)}\bigg)^{-1}-2mn.\nonumber
\end{align}
\end{lemma}
\bigskip

\begin{theorem}\label{decider}
Let $\vec{z}=(z_1,z_2,\ldots,z_n)\in \mathbb{R}^n$ with $z_i\neq z_j$ for all $1\leq i<j\leq n$ with $y_i, z_i\geq 1$ for all $1\leq i\leq n$ and $m:=m(n)=o(1)$ as $n\longrightarrow \infty$. Then  $\vec{z}\in \mathcal{B}_{\frac{1}{2}\mathcal{G}\circ \mathbb{V}_m[\vec{y}]}[\vec{y}]$ with $||\vec{z}||<||\vec{y}||$ if and only if 
\begin{align}
\mathcal{G}\circ \mathbb{V}_m[\vec{z}]\leq \mathcal{G}\circ \mathbb{V}_m[\vec{y}]\nonumber
\end{align}
with $||\vec{y}-\vec{z}||<\epsilon$ for some $\epsilon>0$.
\end{theorem}

\begin{proof}
Let $\vec{z}\in \mathcal{B}_{\frac{1}{2}\mathcal{G}\circ \mathbb{V}_m[\vec{y}]}[\vec{y}]$ for $\vec{z}=(z_1,z_2,\ldots,z_n)\in \mathbb{R}^n$ with $z_i\neq z_j$ for all $1\leq i<j\leq n$ and $z_i\geq 1$ for all $1\leq i\leq n$ such that $||\vec{y}||>||\vec{z}||$. Suppose that 
\begin{align}
\mathcal{G}\circ \mathbb{V}_m[\vec{z}]>\mathcal{G}\circ \mathbb{V}_m[\vec{y}].\nonumber
\end{align}
We deduce $||\vec{y}||\lesssim||\vec{z}||$ which contradicts $||\vec{y}||>||\vec{z}||$. In this case, we can take $\epsilon:=\frac{1}{2}\mathcal{G}\circ \mathbb{V}_m[\vec{y}]$. Conversely, suppose that 
\begin{align}
\mathcal{G}\circ \mathbb{V}_m[\vec{z}]\leq \mathcal{G}\circ \mathbb{V}_m[\vec{y}].\nonumber
\end{align}
We deduce from Proposition \ref{cgidentity} that $||\vec{z}||\lesssim||\vec{y}||$. Under the requirement $||\vec{y}-\vec{z}||<\epsilon$ for some $\epsilon>0$, we obtain the inequality
\begin{align}
\bigg|\bigg|\vec{z}-\frac{1}{2}\bigg(y_1+\frac{m}{y_1},\ldots,y_n+\frac{m}{y_n}\bigg)\bigg|\bigg|&\leq \bigg|\bigg|\vec{y}-\frac{1}{2}\bigg(y_1+\frac{m}{y_1},\ldots,y_n+\frac{m}{y_n}\bigg)\bigg|\bigg|+\epsilon\nonumber \\&=\frac{1}{2}\mathcal{G}\circ \mathbb{V}_m[\vec{y}]+\epsilon \nonumber
\end{align}
with $m=m(n)=o(1)$ as $n\longrightarrow \infty$. Choosing $\epsilon>0$ sufficiently small, we deduce that $\vec{z}\in \mathcal{B}_{\frac{1}{2}\mathcal{G}\circ \mathbb{V}_m[\vec{y}]}[\vec{y}]$.
\end{proof}
\bigskip

In the geometry of balls under compression, we will assume that $n$ is sufficiently large for $\mathbb{R}^n$. In this regime, we will always take the scale of compression $m:=m(n)=o(1)$ as $n\longrightarrow \infty.$

\begin{theorem}\label{ballproof}
Let $\vec{x}=(x_1,x_2,\ldots,x_n)\in \mathbb{R}^n$ with $x_i\neq x_j$ for all $1\leq i<j\leq n$ with $y_i, x_i\geq 1$ for each $1\leq i\leq n$. If $\vec{y}\in \mathcal{B}_{\frac{1}{2}\mathcal{G}\circ \mathbb{V}_m[\vec{x}]}[\vec{x}]$ with $||\vec{y}||<||\vec{x}||$ for $||\vec{y}-\vec{x}||<\delta$ for sufficiently small $\delta>0$. We have
\begin{align}
\mathcal{B}_{\frac{1}{2}\mathcal{G}\circ \mathbb{V}_m[\vec{y}]}[\vec{y}]\subseteq \mathcal{B}_{\frac{1}{2}\mathcal{G}\circ \mathbb{V}_m[\vec{x}]}[\vec{x}]\nonumber
\end{align}
for $m:=m(n)=o(1)$ as $n\longrightarrow \infty.$
\end{theorem}

\begin{proof}
Suppose that $\vec{y}\in \mathcal{B}_{\frac{1}{2}\mathcal{G}\circ \mathbb{V}_m[\vec{x}]}[\vec{x}]$ with $||\vec{y}||<||\vec{x}||$ for $||\vec{y}-\vec{x}||<\delta$. We get from Theorem \ref{decider} that $\mathcal{G}\circ \mathbb{V}_m[\vec{x}]\gtrsim \mathcal{G}\circ \mathbb{V}_m[\vec{y}]$ with $||\vec{y}-\vec{x}||<\delta$ for sufficiently small $\delta>0$. Consequently, the ball $\mathcal{B}_{\frac{1}{2}\mathcal{G}\circ \mathbb{V}_m[\vec{x}]}[\vec{x}]$ is slightly larger than the ball $\mathcal{B}_{\frac{1}{2}\mathcal{G}\circ \mathbb{V}_m[\vec{y}]}[\vec{y}]$ due to their compression gaps, and the latter does not contain the point $\vec{x}$ by construction. We deduce 
$$
||\mathbb{V}_m[\vec{y}]||>||\mathbb{V}_m[\vec{x}]||
$$ 
and 
\begin{align}
\mathcal{G}\circ \mathbb{V}_m[\mathbb{V}_m[\vec{y}]]&=\mathcal{G}\circ \mathbb{V}_m[\vec{y}]\nonumber \\&\lesssim \mathcal{G}\circ \mathbb{V}_m[\vec{x}]\nonumber \\&=\mathcal{G}\circ \mathbb{V}_m[\mathbb{V}_m[\vec{x}]]\nonumber
\end{align}
with $||\mathbb{V}_m[\vec{y}]-\mathbb{V}_m[\vec{x}]||<\epsilon$ for small $\epsilon>0$. It implies 
\begin{align}
\mathcal{B}_{\frac{1}{2}\mathcal{G}\circ \mathbb{V}_m[\vec{y}]}[\vec{y}]\subseteq \mathcal{B}_{\frac{1}{2}\mathcal{G}\circ \mathbb{V}_m[\vec{x}]}[\vec{x}].\nonumber
\end{align} 
\end{proof}

\begin{remark}
Theorem \ref{ballproof} suggests that points in certain balls induced under compression should necessarily have their induced ball under compression nested in this ball.
\end{remark}

\subsection{Interior points and the limit points of balls induced under compression}

In this section, we introduce the notion of an interior and the limit point of balls induced under compression.

\begin{definition}
Let $\vec{y}=(y_1,y_2,\ldots,y_n)\in \mathbb{R}^n$ with $y_i\neq y_j$ for all $1\leq i<j\leq n$. The point $\vec{z}\in \mathcal{B}_{\frac{1}{2}\mathcal{G}\circ \mathbb{V}_m[\vec{y}]}[\vec{y}]$ is an \emph{interior point} if 
\begin{align}
\mathcal{B}_{\frac{1}{2}\mathcal{G}\circ \mathbb{V}_m[\vec{z}]}[\vec{z}]\subseteq \mathcal{B}_{\frac{1}{2}\mathcal{G}\circ \mathbb{V}_m[\vec{x}]}[\vec{x}]\nonumber
\end{align}
for most $\vec{x}\in \mathcal{B}_{\frac{1}{2}\mathcal{G}\circ \mathbb{V}_m[\vec{y}]}[\vec{y}]$. An interior point $\vec{z}$ is then said to be a \emph{limit point} if 
\begin{align}
\mathcal{B}_{\frac{1}{2}\mathcal{G}\circ \mathbb{V}_m[\vec{z}]}[\vec{z}]\subseteq \mathcal{B}_{\frac{1}{2}\mathcal{G}\circ \mathbb{V}_m[\vec{x}]}[\vec{x}]\nonumber
\end{align}
for all $\vec{x}\in \mathcal{B}_{\frac{1}{2}\mathcal{G}\circ \mathbb{V}_m[\vec{y}]}[\vec{y}]$
\end{definition}
\bigskip

We will show that there must exist an interior and limit point in any ball induced by points under compression of any scale on points in $\mathbb{R}^k$.

\begin{theorem}\label{limitexistence}
Let $\vec{x}=(x_1,x_2,\ldots,x_n)\in \mathbb{R}^n$ with $x_i\neq x_j$ for all $1\leq i<j\leq n$ with $y_i\geq 1$ for all $1\leq i\leq n$. The ball $\mathcal{B}_{\frac{1}{2}\mathcal{G}\circ \mathbb{V}_m[\vec{x}]}[\vec{x}]$ contains an interior point and a limit point.
\end{theorem}

\begin{proof}
Let $\vec{x}=(x_1,x_2,\ldots,x_n)\in \mathbb{R}^n$ with $x_i\neq x_j$ for all $1\leq i<j\leq n$ with $x_i\geq 1$ for all $1\leq i\leq n$. Suppose that $\mathcal{B}_{\frac{1}{2}\mathcal{G}\circ \mathbb{V}_m[\vec{x}]}[\vec{x}]$ contains no limit point. Pick 
\begin{align}
\vec{z}_1\in \mathcal{B}_{\frac{1}{2}\mathcal{G}\circ \mathbb{V}_m[\vec{x}]}[\vec{x}].\nonumber
\end{align}
such that $z_{1_i}\geq 1$ for all $1\leq i\leq n$ with $||\vec{z}_1||<||\vec{x}||$ and such that $||\vec{z}_1-\vec{x}||<\epsilon$ for sufficiently small $\epsilon>0$. By Theorem \ref{ballproof} and Theorem \ref{decider}, we deduce
\begin{align}
\mathcal{B}_{\frac{1}{2}\mathcal{G}\circ \mathbb{V}_m[\vec{z}_1]}[\vec{z}_1]\subset \mathcal{B} _{\frac{1}{2}\mathcal{G}\circ \mathbb{V}_m[\vec{x}]}[\vec{x}]\nonumber
\end{align}
with 
$$
\mathcal{G}\circ \mathbb{V}_m[\vec{z}_1]\lesssim \mathcal{G}\circ \mathbb{V}_m[\vec{x}].
$$ 
Again, pick $\vec{z}_2\in \mathcal{B}_{\frac{1}{2}\mathcal{G}\circ \mathbb{V}_m[\vec{z}_1]}[\vec{z}_1]$ so that $z_{2_i}\geq 1$ for all $1\leq i\leq n$ with $||\vec{z}_2||<||\vec{z}_1||$ such that $||\vec{z}_2-\vec{z}_1||<\delta$ for sufficiently small $\delta>0$. By Theorem \ref{ballproof} and Theorem \ref{decider}, we get
\begin{align}
\mathcal{B}_{\frac{1}{2}\mathcal{G}\circ \mathbb{V}_m[\vec{z}_2]}[\vec{z}_2]\subset \mathcal{B}_{\frac{1}{2}\mathcal{G}\circ \mathbb{V}_m[\vec{z}_1]}[\vec{z}_1]\nonumber
\end{align}
with $\mathcal{G}\circ \mathbb{V}_m[\vec{z}_2]\lesssim \mathcal{G}\circ \mathbb{V}_m[\vec{z}_1]$. Continuing the argument in this way, we obtain the infinite descending sequence of the compression gap
\begin{align}
\mathcal{G}\circ \mathbb{V}_m[\vec{x}]\gtrsim \mathcal{G}\circ \mathbb{V}_m[\vec{z}_1]\gtrsim \mathcal{G}\circ \mathbb{V}_m[\vec{z}_2]\gtrsim \cdots \gtrsim \mathcal{G}\circ \mathbb{V}_m[\vec{z}_n]\gtrsim \cdots.\nonumber
\end{align}
\end{proof}

\begin{proposition}\label{found limit point}
The point $\vec{x}=(x_1,x_2,\ldots,x_n)$ with $x_i=1$ for each $1\leq i\leq n$ is the \emph{limit point} of the ball $\mathcal{B}_{\frac{1}{2}\mathcal{G}\circ \mathbb{V}_1[\vec{y}]}[\vec{y}]$ for any $\vec{y}=(y_1,y_2,\ldots,y_n)\in \mathbb{R}^n$ with $y_i>1$ for each $1\leq i\leq n$.
\end{proposition}

\begin{proof}
Applying the compression $\mathbb{V}_1:\mathbb{R}^n\longrightarrow \mathbb{R}^n$ on the point $\vec{x}=(x_1,x_2,\ldots,x_n)$ with $x_i=1$ for each $1\leq i\leq n$, we obtain $\mathbb{V}_1[\vec{x}]=(1,1,\ldots,1)$ so that $\mathcal{G}\circ \mathbb{V}_1[\vec{x}]=0$ and the corresponding ball induced under compression $\mathcal{B}_{\frac{1}{2}\mathcal{G}\circ \mathbb{V}_1[\vec{x}]}[\vec{x}]$ contain only the point $\vec{x}$. It follows that the point $\vec{x}$ must be the limit point of the ball $\mathcal{B}_{\frac{1}{2}\mathcal{G}\circ \mathbb{V}_1[\vec{x}]}[\vec{x}]$. We must have
\begin{align}
 \mathcal{B}_{\frac{1}{2}\mathcal{G}\circ \mathbb{V}_1[\vec{x}]}[\vec{x}] \subseteq  \mathcal{B}_{\frac{1}{2}\mathcal{G}\circ \mathbb{V}_1[\vec{y}]}[\vec{y}]\nonumber 
\end{align}for any $\vec{y}=(y_1,y_2,\ldots,y_n)\in \mathbb{R}^n$ with $y_i>1$ for all $1\leq i\leq n$. Suppose that
\begin{align}
 \mathcal{B}_{\frac{1}{2}\mathcal{G}\circ \mathbb{V}_1[\vec{x}]}[\vec{x}] \not \subseteq  \mathcal{B}_{\frac{1}{2}\mathcal{G}\circ \mathbb{V}_1[\vec{y}]}[\vec{y}]\nonumber    
\end{align}
holds for some $\vec{y}=(y_1,y_2,\ldots,y_n)\in \mathbb{R}^n$ with $y_i>1$ for each $1\leq i\leq n$. There must exist some point $\vec{z}\in \mathcal{B}_{\frac{1}{2}\mathcal{G}\circ \mathbb{V}_1[\vec{x}]}[\vec{x}]$ such that $\vec{z}\not \in \mathcal{B}_{\frac{1}{2}\mathcal{G}\circ \mathbb{V}_1[\vec{y}]}[\vec{y}]$. Since $\vec{x}$ is the only point in the ball $\mathcal{B}_{\frac{1}{2}\mathcal{G}\circ \mathbb{V}_1[\vec{x}]}[\vec{x}]$, we get 
\begin{align}
    \vec{x}\not \in \mathcal{B}_{\frac{1}{2}\mathcal{G}\circ \mathbb{V}_1[\vec{y}]}[\vec{y}]\nonumber
\end{align}
which is inconsistent with the fact that $\vec{x}$ is the limit point of the ball.
\end{proof}

\subsection{Admissible points of balls induced under compression}
We introduce the notion of an \emph{admissible} point of the balls induced by points under compression.

\begin{definition}
Let $\vec{y}=(y_1,y_2,\ldots,y_n)\in \mathbb{R}^n$ with $y_i\neq y_j$ for all $1\leq i<j\leq n$. The point $\vec{y}$ is said to be an \emph{admissible} point of the ball $\mathcal{B}_{\frac{1}{2}\mathcal{G}\circ \mathbb{V}_m[\vec{x}]}[\vec{x}]$ if \begin{align}
\bigg|\bigg|\vec{y}-\frac{1}{2}\bigg(x_1+\frac{m}{x_1},\ldots,x_n+\frac{m}{x_n}\bigg)\bigg|\bigg|=\frac{1}{2}\mathcal{G}\circ \mathbb{V}_m[\vec{x}].\nonumber
\end{align}
\end{definition}
\bigskip

The notion of admissible points of balls induced by points under compression encompasses points on the ball. Geometrically, these points sit on the induced ball.

\begin{theorem}\label{admissibletheorem}
Let $\vec{x}\in \mathbb{R}^n$ with $x_i\neq x_j$~($i\neq j$) such that $x_i, y_i \geq 1$ for all $1\leq i\leq n$ and set $m:=m(n)=o(1)$ as $n\longrightarrow \infty$. The point $\vec{y}\in \mathcal{B}_{\frac{1}{2}\mathcal{G}\circ \mathbb{V}_m[\vec{x}]}[\vec{x}]$ with $||\vec{y}||<||\vec{x}||$ such that $||\vec{y}-\vec{x}||<\epsilon$ for sufficiently small $\epsilon>0$  is admissible if and only if 
\begin{align}
\mathcal{B}_{\frac{1}{2}\mathcal{G}\circ \mathbb{V}_m[\vec{y}]}[\vec{y}]=\mathcal{B}_{\frac{1}{2}\mathcal{G}\circ \mathbb{V}_m[\vec{x}]}[\vec{x}]\nonumber
\end{align}
and $$
\mathcal{G}\circ \mathbb{V}_m[\vec{y}]=\mathcal{G}\circ \mathbb{V}_m[\vec{x}].
$$
\end{theorem}

\begin{proof}
Suppose that $\vec{y}\in \mathcal{B}_{\frac{1}{2}\mathcal{G}\circ \mathbb{V}_m[\vec{x}]}[\vec{x}]$ with $||\vec{y}||<||\vec{x}||$ such that $||\vec{y}-\vec{x}||<\epsilon$ for sufficiently small $\epsilon>0$ is admissible. Assume that 
\begin{align}
\mathcal{B}_{\frac{1}{2}\mathcal{G}\circ \mathbb{V}_m[\vec{y}]}[\vec{y}]\neq \mathcal{B}_{\frac{1}{2}\mathcal{G}\circ \mathbb{V}_m[\vec{x}]}[\vec{x}].\nonumber
\end{align}
Without loss of generality, we can choose some $\vec{z}\in \mathcal{B}_{\frac{1}{2}\mathcal{G}\circ \mathbb{V}_m[\vec{x}]}[\vec{x}]$ with $||\vec{z}||<||\vec{x}||$ such that 
\begin{align}
\vec{z}\notin \mathcal{B}_{\frac{1}{2}\mathcal{G}\circ \mathbb{V}_m[\vec{y}]}[\vec{y}].\nonumber
\end{align}
for $||\vec{z}-\vec{x}||<\delta$ for sufficiently small $\delta>0$. Applying Theorem \ref{decider}, we obtain the inequality 
\begin{align}
\mathcal{G}\circ \mathbb{V}_m[\vec{y}]\lesssim \mathcal{G}\circ \mathbb{V}_m[\vec{x}].\nonumber
\end{align}
This violates equality 
$$
\mathcal{G}\circ \mathbb{V}_m[\vec{y}]=\mathcal{G}\circ \mathbb{V}_m[\vec{x}].
$$  
The latter equality of compression gaps follows from the requirement that the balls are indistinguishable. Conversely, suppose that 
\begin{align}
\mathcal{B}_{\frac{1}{2}\mathcal{G}\circ \mathbb{V}_m[\vec{y}]}[\vec{y}]=\mathcal{B}_{\frac{1}{2}\mathcal{G}\circ \mathbb{V}_m[\vec{x}]}[\vec{x}]\nonumber
\end{align}
and 
$$
\mathcal{G}\circ \mathbb{V}_m[\vec{y}]=\mathcal{G}\circ \mathbb{V}_m[\vec{x}].
$$ 
It follows that the point $\vec{y}$ lives on the outer of the two indistinguishable balls and therefore must satisfy the equality
\begin{align}
\bigg|\bigg|\vec{z}-\frac{1}{2}\bigg(y_1+\frac{m}{y_1},\ldots,y_n+\frac{m}{y_n}\bigg)\bigg|\bigg|&=\bigg|\bigg|\vec{z}-\frac{1}{2}\bigg(x_1+\frac{m}{x_1},\ldots,x_n+\frac{m}{x_n}\bigg)\bigg|\bigg|\nonumber \\&=\frac{1}{2}\mathcal{G}\circ \mathbb{V}_m[\vec{x}].\nonumber
\end{align}
It follows that 
\begin{align}
\frac{1}{2}\mathcal{G}\circ \mathbb{V}_m[\vec{x}]&=\bigg|\bigg|\vec{y}-\frac{1}{2}\bigg(x_1+\frac{m}{x_1},\ldots,x_n+\frac{m}{x_n}\bigg)\bigg|\bigg|\nonumber
\end{align}
and $\vec{y}$ is an admissible point.
\end{proof}
\bigskip

Here, we formulate an equivalent notion of the area of the circle induced by points under compression in the plane $\mathbb{R}^2$.
\bigskip

\begin{proposition}\label{circle area}
Let $\vec{x}\in \mathbb{R}^2$ with $x_i\neq 0$ for each $1\leq i\leq 2$. The area of the circle induced by the point $\vec{x}$ under compression $\mathbb{V}_m[\vec{x}]$ of the scale $m$ is
\begin{align}
\delta(\mathbb{V}_m[\vec{x}])=\frac{\pi(\mathcal{G}\circ \mathbb{V}_m[\vec{x}])^2}{4}.\nonumber
\end{align}
\end{proposition}

\begin{proof}
This follows from the definition of the area of a circle and with the observation that the radius $r$ of the circle induced by the point $\vec{x}\in \mathbb{R}^2$ under compression is  
\begin{align}
r:=\frac{\mathcal{G}\circ \mathbb{V}_m[\vec{x}]}{2}.\nonumber
\end{align}
\end{proof}

\section{The upper bound}\label{sec:upper}
In this section, we deduce an improved upper bound for the problem using the geometry developed in the preceding section.

\begin{theorem}
Let $\Delta(s)$ denote the minimal area of the triangle formed by the $s$ points on the unit disk. We have
\begin{align}
\Delta(s)\ll \frac{1}{s^{\frac{3}{2}-\epsilon}}\nonumber
\end{align}
for small $\epsilon:=\epsilon(s)>0$.
\end{theorem}

\begin{proof}
Let $s\geq 4$ and fix the scale $m$ for $1\geq m:=m(s)>0$. Pick arbitrarily a point $(x_1,x_2)=\vec{x}\in\mathbb{R}^2$ with $x_j>1$ for $1\leq j\leq 2$ so that $x_1\neq x_2$ and set $\mathcal{G}\circ \mathbb{V}_m[\vec{x}]<1$. This ensures that the circle induced under compression is contained in some unit disk. We apply the compression $\mathbb{V}_m[\vec{x}]$ of the scale $m$ with $1\geq m>0$ and construct the circle induced by the compression 
\begin{align}
\mathcal{B}_{\frac{1}{2}\mathcal{G}\circ \mathbb{V}_m[\vec{x}]}[\vec{x}]\nonumber
\end{align}
with radius $\frac{(\mathcal{G}\circ \mathbb{V}_m[\vec{x}])}{2}$. It follows from a simple geometric argument that the smallest area of a triangle formed by the $s$ points on the unit disk (compression circle)
\begin{align}
\Delta(s)&\leq \frac{\pi (\mathcal{G}\circ \mathbb{V}_m[\vec{x}])^2}{4s}\nonumber \\&\ll\frac{2\mathrm{sup}(x_j^2)+m^2\log \bigg(1+\frac{1}{\mathrm{inf}(x_j^2)}\bigg)-4m}{4s}.\nonumber
\end{align}
The upper bound follows taking
\begin{align}m:=\frac{1}{2}+\frac{\log ^2s}{4s}<1 \quad \mathrm{and}\quad \mathrm{sup}(x_j):= 1+\frac{\log s}{\sqrt{s}}\nonumber
\end{align}
since points $\vec{x}=(x_1,x_2)$ can only have a compression gap $\mathcal{G}\circ \mathbb{V}_m[\vec{x}]<1$ if $x_1=1+\delta$ and $x_2=1+\epsilon$ for any small $\delta, \epsilon>0$.
\end{proof}
\bigskip

\section{The lower bound}\label{sec:lower}
In this section, we deduce a lower bound for the Heilbronn triangle problem using the compression geometry developed in the preceding section.

\begin{theorem}
Let $\Delta(s)$ denote the minimal area of the triangle formed by the $s$ points on the unit disk. We have
\begin{align}
\Delta(s)\gg \frac{\log s}{s\sqrt{s}}.\nonumber
\end{align}
\end{theorem}

\begin{proof}
Let $s\geq 4$ and fix the scale $m$ with $1\geq m:=m(s)>0$. Pick arbitrarily a point $(x_1,x_2)=\vec{x}\in\mathbb{R}^2$ with $x_j>1$ for $1\leq j\leq 2$ so that $x_1\neq x_2$ and set $\mathcal{G}\circ \mathbb{V}_m[\vec{x}]<1$. This ensures that the circle induced under compression is contained on some unit disk. Apply the compression $\mathbb{V}_m[\vec{x}]$ of the scale $m$ with $1\geq m>0$ and construct the circle induced by compression 
\begin{align}
\mathcal{B}_{\frac{1}{2}\mathcal{G}\circ \mathbb{V}_m[\vec{x}]}[\vec{x}]\nonumber
\end{align}
with radius $\frac{(\mathcal{G}\circ \mathbb{V}_m[\vec{x}])}{2}$. In this circle locate $(s-3)$ admissible points so that the chord joining each pair of adjacent $(s-1)$ admissible points including $\vec{x}$ and $\mathbb{V}_m[\vec{x}]$ is equidistant. We join each of the $(s-1)$ admissible points considered to the center of the circle 
\begin{align}
\vec{y}:=\frac{1}{2}\bigg(x_1+\frac{m}{x_1},x_2+\frac{m}{x_2}\bigg).\nonumber
\end{align}
By Proposition \ref{circle area}, we deduce the area of the circle induced by compression 
\begin{align}
\delta(\mathbb{V}_m[\vec{x}])=\frac{\pi(\mathcal{G}\circ \mathbb{V}_m[\vec{x}])^2}{4}.\nonumber
\end{align}
We join all pairs of adjacent admissible points considered by a chord and produce $(s-1)$ triangles of equal area. We note that we can use the area of each sector formed from this construction to approximate the area of each of the triangles inscribed in the sector as we increase the number of such admissible points on the circle. We deduce
\begin{align}
\mathcal{A}:&=\frac{\pi (\mathcal{G}\circ \mathbb{V}_m[\vec{x}])^2}{4\times (s-1)}\nonumber \\&\gg \frac{2\mathrm{inf}(x_j^2)+m^2\log \bigg(1-\frac{1}{\mathrm{sup}(x_j^2)}\bigg)^{-1}-4m}{4\times s}.\nonumber
\end{align}
The lower bound is obtained by taking
\begin{align}
m:=\frac{1}{2}+\frac{\log ^2s}{4s}<1 \quad \mathrm{and}\quad \mathrm{inf}(x_j):= 1+\frac{\log s}{\sqrt{s}}\nonumber
\end{align}
since points $\vec{x}=(x_1,x_2)$ can only have a compression gap $\mathcal{G}\circ \mathbb{V}_m[\vec{x}]<1$ if $x_1=1+\delta$ and $x_2=1+\epsilon$ for any small $\delta,\epsilon>0$.
\end{proof}

\footnote{
\par
.}%

\bibliographystyle{amsplain}

\end{document}